\documentclass[reqno, 11pt, letterpaper]{amsart}

\usepackage{amsmath}
\usepackage{amsfonts}
\usepackage{amssymb}
\usepackage{graphicx}
\usepackage{amsthm,color,yfonts,cite} \usepackage{paralist}
\usepackage{hyperref}

\newtheorem{theorem}{Theorem}

\newtheorem{proposition}[theorem]{Proposition}
\newtheorem{lemma}[theorem]{Lemma}
\newtheorem{corollary}[theorem]{Corollary}

\theoremstyle{remark}
\newtheorem{remark}[theorem]{Remark}

\makeatletter
\newcommand*{\rom}[1]{\expandafter\@slowromancap\romannumeral #1@}
\makeatother

\newcommand{\R}{\mathbb{R}}

\usepackage{wrapfig}
\usepackage{tikz}
\usetikzlibrary{arrows,calc,decorations.pathreplacing}
\definecolor{light-gray1}{gray}{0.90}
\definecolor{light-gray2}{gray}{0.80}
\definecolor{light-gray3}{gray}{0.60}

\numberwithin{equation}{section}

\numberwithin{theorem}{section}

\numberwithin{table}{section}

\numberwithin{figure}{section}

\ifx\pdfoutput\undefined
  \DeclareGraphicsExtensions{.pstex, .eps}
\else
  \ifx\pdfoutput\relax
    \DeclareGraphicsExtensions{.pstex, .eps}
  \else
    \ifnum\pdfoutput>0
      \DeclareGraphicsExtensions{.pdf}
    \else
      \DeclareGraphicsExtensions{.pstex, .eps}
    \fi
  \fi
\fi

\title[Orbital stability for the mass-(super)critical pseudo-relativistic NLS]{Orbital stability for the mass-critical and supercritical pseudo-relativistic nonlinear Schr\"odinger equation}

\date{\today}
\author[Y. Hong]{Younghun Hong}
\address{Department of Mathematics, Chung-Ang University, Seoul 06974, Korea}
\email{yhhong@cau.ac.kr}

\author[S. Jin]{Sangdon Jin}
\address{Department of Mathematics, Chung-Ang University, Seoul 06974, Korea}
\email{sdjin@cau.ac.kr}

\begin{document}

\begin{abstract}

For the one-dimensional mass-critical/supercritical pseudo-relativistic nonlinear Schr\"odinger equation, a stationary solution can be constructed as an energy minimizer under an additional kinetic energy constraint and the set of energy minimizers is orbitally stable \cite{BGV}. In this study, we proved the local uniqueness and established the orbital stability of the solitary wave by improving that of the energy minimizer set. A key aspect thereof is the reformulation of the variational problem in the non-relativistic regime, which we consider to be more natural because the proof extensively relies on the subcritical nature of the limiting model. Thus, the role of the additional constraint is clarified, a more suitable Gagliardo-Nirenberg inequality is introduced, and the non-relativistic limit is proved. Subsequently, this limit is employed to derive the local uniqueness and orbital stability. 
\end{abstract}

\maketitle


\section{Introduction}

We consider the pseudo-relativistic nonlinear Schr\"odinger equation (NLS)
\begin{equation}\label{pNLS0}
i\partial_t u=\left(\sqrt{m^2c^4-c^2\partial_x^2}-mc^2\right)u-|u|^{p-1}u,
\end{equation}
where $u=u(t,x):I(\subset\mathbb{R})\times\mathbb{R}\to\mathbb{C}$. The operator $\sqrt{m^2c^4-c^2\partial_x^2}-mc^2$ is defined as a Fourier multiplier. The constant $c$ represents the speed of light, and $m$ is the particle mass. This operator is called \textit{pseudo-relativistic} or \textit{semi-relativistic} because it describes the intermediate dynamics between the non-relativistic regime $c\gg1$ and the ultra-relativistic regime $0<c\ll1$. When the power-type nonlinearity is replaced by Hartree nonlinearity, the equation is referred to as the boson star equation \cite{FJL, Lenzmann0}.

We suppose that $1<p<5$. Then, the Cauchy problem for \eqref{pNLS0} is locally well-posed in $H^{1/2}(\mathbb{R})$ (see \cite{BR, KLR, FMO, SP})\footnote{The higher-dimensional cases are considered in \cite{BGV}, but the radially symmetric assumption is imposed. }, and the solutions preserve the mass
$$\mathcal{M}(u)=\int_{\mathbb{R}}|u(x)|^2dx$$
and the energy
$$\mathcal{E}_{m,c}(u)=\frac{1}{2}\int_{\mathbb{R}}\left(\sqrt{m^2c^4-c^2\partial_x^2}-mc^2\right)u(x)\overline{u(x)} dx-\frac{1}{p+1}\int_{\mathbb{R}}|u(x)|^{p+1}dx.$$
When the nonlinearity is mass-subcritical, that is, $1<p<3$, the equation \eqref{pNLS0} is globally well-posed in $H^{\frac{1}{2}}(\mathbb{R})$, and the corresponding energy minimization problem
\begin{equation}\label{minimization0}
\mathcal{J}(M):=\inf\Big\{\mathcal{E}_{m,c}(u): u\in H^{\frac{1}{2}}(\mathbb{R})\textup{ and }\mathcal{M}(u)=M\Big\}
\end{equation}
yields an orbitally stable ground state (see \cite{BGV, CL}). In contrast, in the mass-critical or supercritical case $p\in[3,5)$, a blow-up is expected to occur (see \cite{KLR, BHL}), and the variational problem \eqref{minimization0} is not appropriately formulated because $\mathcal{J}(M)=-\infty$.

The meaning of the mass-critical nonlinearity is however ambiguous because of its semi-relativistic nature. In the non-relativistic regime $c\gg1$, equation \eqref{pNLS0} is formally approximated by the mass-subcritical non-relativistic NLS 
$$i\partial_t u=-\tfrac{1}{2m}\partial_x^2 u-|u|^{p-1}u.$$
Based thereupon, in the non-relativistic regime, an orbitally stable state can be constructed from a modified energy minimization problem, which is one of the main observations by Bellazzini, Georgiev, and Visciglia \cite{BGV} (see also \cite{BBJV} for orbitally stable ground states to NLS with a partial
confinement). Precisely, the energy minimization problem with an additional constraint 
\begin{equation}\label{BGV minimization}
\inf\Big\{\mathcal{E}_{1,1}(u): \|u\|_{H^{\frac{1}{2}}(\mathbb{R})}\leq\tfrac{1}{2}\textup{ and }\mathcal{M}(u)=M\Big\}.
\end{equation}
Furthermore, the existence of an energy minimizer is established, provided that the mass $M$ is sufficiently small. Then, using the standard argument \cite{CL, Lions}, the orbital stability of the set of local energy minimizers follows.

In this study, we revisit the variational problem \eqref{BGV minimization}, but we rescale it to clarify the connection to the non-relativistic problem, as in the work by Lenzmann \cite{Lenzmann}. Hereafter, we take $m=\frac{1}{2}$ for numerical simplicity and denote 
$$\mathcal{H}_c:=\sqrt{-c^2\partial_x^2+\tfrac{c^4}{4}}-\tfrac{c^2}{2}\quad\textup{and}\quad \mathcal{E}_c(u):=\mathcal{E}_{\frac{1}{2},c}(u)=\frac{1}{2}\|\sqrt{\mathcal{H}_c}u\|_{L^2(\mathbb{R})}^2-\frac{1}{p+1}\|u\|_{L^{p+1}(\mathbb{R})}^{p+1}.$$
Then, instead of fixing $c=1$, we choose $M>0$ (where $M$ is not necessarily small) and consider the energy minimization problem 
\begin{equation}\label{pNLS minimization}
\mathcal{J}_c(M):=\inf\Big\{\mathcal{E}_c(u): \|\sqrt{\mathcal{H}_c}u\|_{L^2(\mathbb{R})}\leq c^{\frac{p+3}{2(p-1)}} \textup{ and }\mathcal{M}(u)=M\Big\},
\end{equation}
which is equivalent to \eqref{BGV minimization} by scaling (see Section \ref{sec:scaling}). Using the formal series expansion $\sqrt{-c^2\partial_x^2+\frac{c^4}{4}}-\frac{c^2}{2}=-\partial_x^2-\tfrac{1}{c^2}(\partial_x^2)^2+\cdots$, via the non-relativistic limit $c\to\infty$, 
the variational problem \eqref{pNLS minimization} is closely related to the mass-subcritical non-relativistic energy minimization problem
\begin{equation}\label{non-relativistic variational problem}
\mathcal{J}_\infty(M):=\inf\Big\{\mathcal{E}_\infty(u): u\in H^1(\mathbb{R})\textup{ and }\mathcal{M}(u)=M\Big\},
\end{equation}
where
\begin{equation}\label{non-relativistic energy}
\mathcal{E}_\infty(u):=\frac{1}{2}\|\partial_x u\|_{L^2(\mathbb{R})}^2-\frac{1}{p+1}\|u\|_{L^{p+1}(\mathbb{R})}^{p+1}.
\end{equation}
Although this connection was previously employed in \cite{BGV}, it is rather implicit.

Motivated by the formal convergence, we state the existence of a minimizer for the variational problem \eqref{pNLS minimization} as follows.
\begin{theorem}[Minimizer for $\mathcal{J}_c(M)$ ]\label{cptness}
Let $p\in [3,5)$, and suppose that
$$c\geq \max\{(\alpha M)^\frac{p-1}{5-p}, (\alpha^\frac{4}{p+3}M)^\frac{p-1}{5-p}\},$$
where $\alpha$ is given by \eqref{alpha}. Then, the minimization problem $\mathcal{J}_c(M)$ possesses a minimizer, which must be of the following form:
$$e^{i\theta}Q_c(\cdot-x_0)\quad\textup{with }x_0\in\mathbb{R},\ \theta\in\mathbb{R},$$
where $Q_c$ is non-negative symmetric, and it solves the elliptic equation
\begin{equation}\label{pNLS elliptic}
\mathcal{H}_cQ_c-Q_c^p=-\mu_cQ_c.
\end{equation}
Moreover, if $3<p<5$ and $c\geq \max\{(\alpha M)^\frac{p-1}{5-p}, (\alpha^\frac{4}{p+3}M)^\frac{p-1}{5-p},(\frac{M}{p-3})^\frac{p-1}{5-p}\}$, then $Q_c$ is a ground state on $\mathcal{N}$ in the sense that under the mass constraint $\mathcal{M}(v)=M$, it occupies the least energy among solutions to the nonlinear elliptic equation $\mathcal{H}_cu-u^p=-\lambda u$ for some $\lambda>0$.\footnote{Precisely, $
\mathcal{E}_c(Q_c)=\inf \big\{\mathcal{E}_c(u): \mathcal{E}_c|_\mathcal{N}'(u)=0 \big\}$, where $\mathcal{N}=\big\{v\in H^\frac12(\R): \mathcal{M}(v)=M\big\}$.}
\end{theorem}

\begin{remark}
\begin{enumerate}
\item Theorem \ref{cptness} was essentially proved in \cite{BGV}. The difference is that the assumption of a small mass $M\ll1$ \cite{BGV} is transferred to the non-relativistic regime assumption $c\gg1$. Indeed, it could be regarded as simply being a matter of scaling, but it is helpful to clarify the setup of the problem and to improve our intuition. This leads us to modify the Gagliardo-Nirenberg inequality, which is a more appropriate fit for the pseudo-relativistic operator (Proposition \ref{GN type inequality}), but the constraint in the minimization problem is also refined (Lemma \ref{local min structure}). Consequently, we can obtain several useful properties of the minimizer, which can be employed to prove our main result.
\item A similar existence theorem was established for the mass-critical/supercritical pseudo-relativistic Hartree-type equation \cite{LY1}. 
\end{enumerate}
\end{remark}

Our main result provides the uniqueness and the orbital stability of the ground state.

\begin{theorem}\label{thm: uniqueness}
Let $p\in [3,5)$. For any sufficiently large $c\geq1$, let $Q_c$ be an energy minimizer for $\mathcal{J}_c(M)$ constructed in Theorem \ref{cptness}.
\begin{enumerate}[(i)]
\item $Q_c$ is unique up to translation and phase shift. 
\item For $\epsilon>0$, there exists $\delta>0$ such that if 
$$\inf_{x_0,\theta_0\in\mathbb{R}} \|\sqrt{1+\mathcal{H}_c}(e^{i\theta_0}u_0(x-x_0)-Q_c(x))\|_{L^2(\R)}\leq\delta,$$
then the global solution $u(t)$ to
\begin{equation}\label{pNLS}
i\partial_t u=\left(\sqrt{\tfrac{c^4}{4}-c^2\partial_x^2}-\tfrac{c^2}{2}\right)u-|u|^{p-1}u
\end{equation}
with the initial data $u_0$ satisfies
$$\inf_{x_1,\theta_1\in\mathbb{R}}\big\|\sqrt{1+\mathcal{H}_c}(e^{i\theta_1}u(t,x-x_1)-Q_c(x) )\big\|_{L^2(\R)}\leq\epsilon \quad\textup{for all }t\in\mathbb{R}.$$
\end{enumerate}
\end{theorem}

\begin{remark}
Theorem \ref{thm: uniqueness} $(i)$ can be extended to the multidimensional case $d\geq 2$ with power-type nonlinearity $-|u|^{p-1}u$ with $1+\frac{2}{d}<p<1+\frac{4}{d}$. However, because the well-posedness of the time-dependent problem \eqref{pNLS0} is known under the radial assumption, only the orbital stability against symmetric perturbations can be obtained (see \cite{BGV}).
\end{remark}
\begin{remark}
By using scaling in Section \ref{sec:scaling}, for small $M>0$, one can find a minimizer $u_M$ for \eqref{BGV minimization}. Then, our main result can be reformulated as follows: for $\epsilon>0$, there exists $\delta>0$ such that if 
$$\inf_{x_0,\theta_0\in\mathbb{R}} \| e^{i\theta_0}v_0(x-x_0)-u_M(x) \|_{H^{\frac{1}{2}}(\R)}\leq\delta,$$then the global solution $v(t)$ to
\begin{equation*} i\partial_t v=\sqrt{1- \partial_x^2}v-|v|^{p-1}v
\end{equation*}
with the initial data $v_0$ satisfies
$$\inf_{x_1,\theta_1\in\mathbb{R}}\big\| e^{i\theta_1}v(t,x-x_1)-u_M(x) \big\|_{H^{\frac{1}{2}}(\R)}\leq\epsilon\quad\textup{for all }t\in\mathbb{R}.$$
Note that, compared with the previous result \cite{BGV}, the possibility of transforming into a different low-energy state is eliminated.
\end{remark}

To prove the main theorem, we employ a connection to the non-relativistic variational problem $\mathcal{J}_\infty(M)$. Let $Q_\infty$ be a symmetric, positive decreasing ground state for $\mathcal{J}_\infty(M)$. This minimizer is known to be unique up to translation and phase shift and solves the nonlinear elliptic equation 
$$-\partial_x^2 Q_\infty-Q_\infty^p=-\mu_\infty Q_\infty,$$
where $\mu_\infty$ is the Lagrange multiplier. A key step in our analysis is to show the convergence from relativistic to non-relativistic minimizers.

\begin{proposition}[Non-relativistic limit]\label{prop: non-relativistic limit}
Let $p\in [3,5)$. For a sufficiently large $c\geq 1$, let $Q_c$ be the ground state constructed in Theorem \ref{cptness}. Then, $Q_c\to Q_\infty$ in $H^1(\mathbb{R})$ and $\mu_c\to\mu_\infty$ as $c\to\infty$.
\end{proposition}

By the convergence $Q_c\to Q_\infty$, we obtain a certain coercivity estimate for the linearized operator at $Q_c$ (see Lemma \ref{eigv}). Then, using this estimate, we deduce a contradiction by comparing the energies, provided that there are two minimizers (see Section \ref{unq}).

\begin{remark}
In the mass-critical/supercritical case, a radially symmetric solution to \eqref{pNLS elliptic} is constructed near $Q_\infty$ in \cite{CHS} via the contraction mapping argument. The convergence $Q_c\to Q_\infty$ and the local uniqueness show that it is identified with the solution obtained by the variational method. 
\end{remark}

\subsection{Organization of this paper} 
The remainder of this paper is organized as follows. In Section \ref{pre}, we provide preliminaries for the proof of the main results. This section describes a scaling property, a modified Gagliardo-Nirenberg inequality, well-posedness of the pseudo-relativistic NLS, and non-relativistic minimization problem. In Section \ref{psm}, we prove Theorem \ref{cptness} using the modified Gagliardo-Nirenberg inequality and the results for the non-relativistic minimization problem. In Section \ref{nlim}, we present the non-relativistic limit of a minimizer $Q_c$ for $\mathcal{J}_c(M)$ (Proposition \ref{prop: non-relativistic limit}). In Section \ref{unq}, we establish the uniqueness of minimizer $Q_c$ for $\mathcal{J}_c(M)$ (Theorem \ref{thm: uniqueness}). Finally, in Appendix \ref{sec: ground state}, for completeness of the paper, we present that a minimizer $Q_c$ is a ground state solution as stated in Theorem \ref{cptness}.

\subsection{Acknowledgment}
This research was supported by the Basic Science Research Program through the National Research Foundation of Korea (NRF) funded by the Ministry of Science and ICT (NRF-2020R1A2C4002615).

\section{Preliminaries}\label{pre}

\subsection{Scaling}\label{sec:scaling}
By a simple change of the variable, the parameter $c$ in the variational problem \eqref{pNLS minimization} is scaled as follows. Let $u(x)= c^{\frac{2}{p-1}}v(c x)$. Then, we have 
$$\begin{aligned}
\mathcal{M}(u)&= c^{ \frac{5-p}{p-1}} \mathcal{M}(v),\\
\mathcal{E}_{m,c}(u)&= c^{\frac{p+3}{p-1}}\mathcal{E}_{m,1}(v),\\
\|\sqrt{\mathcal{H}_{m,c}}u\|_{L^2(\mathbb{R})}&= c^{\frac{p+3}{2(p-1)}}\|\sqrt{\mathcal{H}_{m,1}}v\|_{L^2(\mathbb{R})},
\end{aligned}$$
where $\mathcal{H}_{m,c}=\sqrt{m^2c^4-c^2\partial_x^2}-mc^2$. Thus, the minimization problem 
$$\inf\left\{\mathcal{E}_{m,c}(u): \|\sqrt{\mathcal{H}_{m,c}}u\|_{L^2(\mathbb{R})}\leq  c^{\frac{p+3}{2(p-1)}} \textup{ and }\mathcal{M}(u)=M\right\}$$
is equivalent to
$$\inf\left\{\mathcal{E}_{m,1}(v): \|\sqrt{\mathcal{H}_{m,1}}v\|_{L^2(\mathbb{R})}\leq 1 \textup{ and }\mathcal{M}(v)=c^{-\frac{5-p}{p-1}}M\right\}.$$

\subsection{Modified Gagliardo-Nirenberg inequality}
The Gagliardo-Nirenberg inequality states that, for $p\geq 1$ with $\frac{1}{p+1}>\frac{1-2s}{2}$, 
\begin{equation}\label{standard GN inequality}
\|u\|_{L^{p+1}(\mathbb{R})}^{p+1}\leq C_{s,p+1}\|u\|_{L^2(\mathbb{R})}^{\frac{(2s-1)p+(2s+1)}{2s}} \||\partial_x|^su\|_{L^2(\mathbb{R})}^{\frac{p-1}{2s}},
\end{equation}
where $u\in L^2(\R)$, $|\partial_x|^s u\in L^2(\R)$ and $C_{s,p+1}$ is a sharp constant. For the pseudo-relativistic NLS and its non-relativistic limit, this inequality is modified considering that the pseudo-relativistic operator acts differently on low and high frequencies.

\begin{proposition}[Modified Gagliardo-Nirenberg inequality]\label{GN type inequality}
For $p>1$, let
$$C_{GN}=2^{\frac{3p-1}{2}}\max\{C_{1,p+1}, C_{\frac{1}{2},p+1}\},$$
where $C_{s,p+1}$ is the sharp constant for the Gagliardo-Nirenberg inequality \eqref{standard GN inequality}. Then, for any $c\geq 1$ and $\delta\in(0,1]$, we have
$$\|u\|_{L^{p+1}(\mathbb{R})}^{p+1}\leq C_{GN} \left\{\|u\|_{L^2(\mathbb{R})}^{\frac{p+3}{2}}\|\sqrt{\mathcal{H}_c} P_{\leq c\delta}u\|_{L^2(\mathbb{R})}^{\frac{p-1}{2}}+\frac{1}{(c\delta)^{\frac{p-1}{2}}}  \|u\|_{L^2(\mathbb{R})}^2\|\sqrt{\mathcal{H}_c} P_{>c\delta}u\|_{L^2(\mathbb{R})}^{p-1}\right\},$$
where $P_{\leq c\delta }u=(\mathbf{1}_{|\cdot|\leq c\delta }\hat{u})^\vee$ and $P_{>c\delta}u=u-P_{\leq c\delta}u$.
\end{proposition}


The following bounds for the symbol are useful. 

\begin{lemma}[Symbol of the pseudo-relativistic operator]\label{symbol}
For $\delta\in(0,1]$, we have
$$\sqrt{c^2|\xi|^2+\tfrac{c^4}{4}}-\tfrac{c^2}{2}\geq\left\{\begin{aligned}
&\tfrac{1}{2}|\xi|^2&&\textup{if }|\xi|\leq c\delta,\\
&\tfrac{c\delta}{2}|\xi|&&\textup{if }|\xi|\geq c\delta.
\end{aligned}\right.$$
and
$$\sqrt{c^2|\xi|^2+\tfrac{c^4}{4}}-\tfrac{c^2}{2}\leq |\xi|^2\quad \textup{for all }\xi.$$
\end{lemma}

\begin{proof}
Because $\sqrt{c^2|\xi|^2+\frac{c^4}{4}}-\frac{c^2}{2}=\frac{c^2}{2}f(\frac{4|\xi|^2}{c^2})$, where $f(t)=\sqrt{1+t}-1$, it suffices to show the appropriate bounds for $f(t)$. Evidently, $f(t)\leq \frac{t}{2}$. Moreover, if $0\leq t\leq 3\delta^2$, then by the mean-value theorem, there exists $t_*\in[0,3\delta^2]$ such that $f(t)=\frac{1}{2\sqrt{1+t_*}}t\ge \frac{1}{4} t$. On the other hand, if $t\geq3\delta^2$, then $f(t)=\frac{t}{\sqrt{t+1}+1}\geq\frac{\delta}{2}\sqrt{t}$ because $\frac{\sqrt{t}}{\sqrt{t+1}+1}$ increases. 
\end{proof}

\begin{proof}[Proof of Proposition \ref{GN type inequality}]
Decomposing the high and low frequencies and then applying the Gagliardo-Nirenberg inequality, we obtain 
$$\begin{aligned}
\|u\|_{L^{p+1}(\mathbb{R})}^{p+1}&\leq 2^p\left\{\|P_{\leq c\delta}u\|_{L^{p+1}(\mathbb{R})}^{p+1}+\|P_{> c\delta}u\|_{L^{p+1}(\mathbb{R})}^{p+1}\right\}\\
&\leq 2^p\left\{ C_{1,p+1}\|u\|_{L^2(\mathbb{R})}^{\frac{p+3}{2}}\||\partial_x| P_{\leq c\delta}u\|_{L^2(\mathbb{R})}^{\frac{p-1}{2}}+C_{\frac{1}{2},p+1}\|u\|_{L^2(\mathbb{R})}^2\||\partial_x|^{1/2}P_{>c\delta}u\|_{L^2(\mathbb{R})}^{p-1}\right\}
\end{aligned}$$
Hence, using the lower bound in Lemma \ref{symbol}, we complete the proof.
\end{proof}

Proposition \ref{GN type inequality} and the Cauchy-Schwarz inequality imply that $\mathcal{J}_c(M)$ has a lower bound.
\begin{corollary} Let $p\in[3,5)$. Then, $\mathcal{J}_c(M)>-\infty$ for all $c\ge 1$.
\end{corollary}

Throughout this paper, we let 
\begin{equation}\label{alpha}
\boxed{\quad\alpha=\frac{4C_{GN}}{p+1}=\frac{ {2}^\frac{3(p+1)}{2}}{p+1}\max\Big\{C_{1,p+1}, C_{\frac{1}{2},p+1}\Big\},\quad}
\end{equation}
where $C_{s,p+1}$ is the sharp constant for the Gagliardo-Nirenberg inequality \eqref{standard GN inequality}.

Another important consequence of Proposition \ref{GN type inequality} is the separation of the negative energy function set for the pseudo-relativistic NLS. 

\begin{corollary}[Separation of the negative energy function space]\label{separation of initial data}
Let $p\in[3,5)$ and $c\geq \max\{(\alpha M)^\frac{p-1}{5-p}, (\alpha^\frac{4}{p+3}M)^\frac{p-1}{5-p}\}$, where $\alpha$ is given by \eqref{alpha}. Suppose that $\mathcal{M}(u)=M$ and $\mathcal{E}_c(u)<0$. Then, either $\|\sqrt{\mathcal{H}_c} u\|_{L^2(\mathbb{R})}> c^{\frac{p+3}{2(p-1)}}$ or $\|\sqrt{\mathcal{H}_c} u\|_{L^2(\mathbb{R})}\leq \alpha^{\frac{2}{5-p}}M^{\frac{p+3}{2(5-p)}}$ holds.
\end{corollary}

\begin{remark}\label{remark: sepration}
In other words, if $\mathcal{M}(u)=M$, $\mathcal{E}_c(u)<0$ and $\|\sqrt{\mathcal{H}_c} u\|_{L^2(\mathbb{R})}\leq c^{\frac{p+3}{2(p-1)}}$, then $\|\sqrt{\mathcal{H}_c} u\|_{L^2(\mathbb{R})}\leq \alpha^{\frac{2}{5-p}}M^{\frac{p+3}{2(5-p)}}$.
\end{remark}

\begin{proof}[Proof of Corollary \ref{separation of initial data}]
Note that, because $c\geq (\alpha^\frac{4}{p+3}M)^\frac{p-1}{5-p}$, we have $c^{\frac{p+3}{2(p-1)}}\ge \alpha^{\frac{2}{5-p}}M^{\frac{p+3}{2(5-p)}}$.
By Proposition \ref{GN type inequality} with $\delta=1$, we write 
$$\mathcal{E}_c(u)\geq\frac{1}{2}\|\sqrt{\mathcal{H}_c} u\|_{L^2(\mathbb{R})}^2 -\frac{\alpha}{4}\left\{\|u\|_{L^2(\mathbb{R})}^{\frac{p+3}{2}}\|\sqrt{\mathcal{H}_c}u\|_{L^2(\mathbb{R})}^{\frac{p-1}{2}}+c^{-\frac{p-1}{2}}\|u\|_{L^2(\mathbb{R})}^2\|\sqrt{\mathcal{H}_c}u\|_{L^2(\mathbb{R})}^{p-1}\right\}.$$
If $\|\sqrt{\mathcal{H}_c} u\|_{L^2(\mathbb{R})}\leq c^{\frac{p+3}{2(p-1)}}$, then by the assumptions, it follows that 
\begin{equation}\label{ene1}
\begin{aligned}
\mathcal{E}_c(u)&\geq\frac{1}{2}\|\sqrt{\mathcal{H}_c} u\|_{L^2(\mathbb{R})}^2 -\frac{\alpha}{4}\left\{M^{\frac{p+3}{4}}\|\sqrt{\mathcal{H}_c}u\|_{L^2(\mathbb{R})}^{\frac{p-1}{2}}+M c^{-\frac{5-p}{p-1}}\|\sqrt{\mathcal{H}_c}u\|_{L^2(\mathbb{R})}^2\right\}\\
&=\big(\tfrac{1}{2}-\tfrac{\alpha}{4}c^{-\frac{5-p}{p-1}}M\big)\|\sqrt{\mathcal{H}_c}u\|_{L^2(\mathbb{R})}^{\frac{p-1}{2}}\left\{\|\sqrt{\mathcal{H}_c}u\|_{L^2(\mathbb{R})}^{\frac{5-p}{2}}-\tfrac{\alpha M^{\frac{p+3}{4}}}{4\left(\frac{1}{2}-\frac{\alpha}{4}c^{-\frac{5-p}{p-1}} M\right)}\right\}\\
&\geq\frac{1}{4}\|\sqrt{\mathcal{H}_c}u\|_{L^2(\mathbb{R})}^{\frac{p-1}{2}}\left\{\|\sqrt{\mathcal{H}_c}u\|_{L^2(\mathbb{R})}^{\frac{5-p}{2}}-\alpha M^{\frac{p+3}{4}}\right\},
\end{aligned}
\end{equation}
where $c\geq (\alpha M)^\frac{p-1}{5-p}$ in the previous step. Hence, because $\mathcal{E}_c(u)<0$, we conclude that $\|\sqrt{\mathcal{H}_c}u\|_{L^2(\mathbb{R})}\leq\alpha^{\frac{2}{5-p}}M^{\frac{p+3}{2(5-p)}}$.
\end{proof}

\subsection{Well-posedness of the pseudo-relativistic NLS}
Even in the mass-(super)critical case, if $c\geq 1$ is sufficiently large, a negative energy solution exists globally in time, and its kinetic energy is not extremely large. 

\begin{proposition}[Global well-posedness]\label{GWP}
Let $p\in[3,5)$ and 
$$c\geq \max\{(\alpha M)^\frac{p-1}{5-p}, (\alpha^\frac{4}{p+3}M)^\frac{p-1}{5-p}\},$$
where $\alpha$ is given by \eqref{alpha}. Suppose that $\mathcal{M}(u_0)=M$, $\mathcal{E}_c(u_0)<0$, and $\|\sqrt{\mathcal{H}_c} u_0\|_{L^2(\mathbb{R})}\leq c^{\frac{p+3}{2(p-1)}}$, and let $u(t)$ be the local solution to the pseudo-relativistic NLS \eqref{pNLS0} with initial data $u_0$. Then, $u(t)$ exists globally in time, and 
$$\sup_{t\in\mathbb{R}}\|\sqrt{\mathcal{H}_c} u(t)\|_{L^2(\mathbb{R})}\leq \alpha^{\frac{2}{5-p}}M^{\frac{p+3}{2(5-p)}}.$$
\end{proposition}

\begin{proof}
Because $\|\sqrt{\mathcal{H}_c} u_0\|_{L^2(\mathbb{R})}\leq c^{\frac{p+3}{2(p-1)}}$ and $u_0$ has a negative energy, it follows from Remark \ref{remark: sepration} that $\|\sqrt{\mathcal{H}_c} u_0\|_{L^2(\mathbb{R})}\leq \alpha^{\frac{2}{5-p}}M^{\frac{p+3}{2(5-p)}}$. Hence, by continuity, $\|\sqrt{\mathcal{H}_c} u(t)\|_{L^2(\mathbb{R})}\leq c^{\frac{p+3}{2(p-1)}}$ for all sufficiently small $|t|$, but from Remark \ref{remark: sepration}, the solution obeys $\|\sqrt{\mathcal{H}_c} u(t)\|_{L^2(\mathbb{R})}\leq \alpha^{\frac{2}{5-p}}M^{\frac{p+3}{2(5-p)}}$ for all small $|t|$. Repeating this analysis, we conclude that the same bound holds for all $t$.
\end{proof}

\subsection{Non-relativistic minimization problem}\label{sec: NR minimization}
We summarize the known results for the limit case \eqref{non-relativistic variational problem}. The variational problem \eqref{non-relativistic variational problem} attains a positive minimizer $Q_\infty$. This minimizer is unique up to translation and phase shift, and it solves the Euler-Lagrange equation $-\partial_x^2 Q_\infty-Q_\infty^p=-\mu_\infty Q_\infty$ for some Lagrange multiplier $\mu_\infty>0$. Moreover, it is symmetric, smooth, and exponentially decreasing (see \cite{Caz, Lions}). Using the following identities, 
$$0=\|\partial_x Q_\infty\|_{L^2(\mathbb{R})}^2-\|Q_\infty\|_{L^{p+1}(\mathbb{R})}^{p+1}+\mu_\infty \|Q_\infty\|_{L^2(\mathbb{R})}^2$$
and
$$\begin{aligned}
0&=2\int_{\mathbb{R}} x\partial_x Q_\infty (-\partial_x^2 Q_\infty-Q_\infty^p+\mu_\infty Q_\infty)dx\\
&=\|\partial_x Q_\infty\|_{L^2(\mathbb{R})}^2+\frac{2}{p+1}\|Q_\infty\|_{L^{p+1}(\mathbb{R})}^{p+1}-\mu_\infty \|Q_\infty\|_{L^2(\mathbb{R})}^2,
\end{aligned}$$
we find that 
\begin{equation}\label{Pohozaev identities}
\|\partial_x Q_\infty\|_{L^2(\mathbb{R})}^2=\frac{p-1}{p+3}\mu_\infty M,\quad \|Q_\infty\|_{L^{p+1}(\mathbb{R})}^{p+1}=\frac{2(p+1)}{p+3} \mu_\infty M,
\end{equation}
where $M=\mathcal{M}(Q_\infty)$.
In \cite{We}, it was  proved that $Q_\infty$ provides a sharp constant for the Gagliardo-Nirenberg inequality (see \eqref{standard GN inequality})
$$C_{1,p+1}=\frac{\|Q_\infty\|_{L^{p+1}(\mathbb{R})}^{p+1}}{\|Q_\infty\|_{L^2(\mathbb{R})}^\frac{p+3}{2}\|\partial_x Q_\infty\|_{L^2(\mathbb{R})}^\frac{p-1}{2}}=\frac{2(p+1)}{(p+3)^{\frac{5-p}{4}}(p-1)^{\frac{p-1}{4}}}\frac{\mu_\infty^\frac{5-p}{4}}{M^\frac{p-1}{2}},$$
where \eqref{Pohozaev identities} are used in the final step. Then, inserting $\mu_\infty=(p+3)[\frac{(p-1)^{\frac{p-1}{4}}C_{1,p+1}}{2(p+1)}]^\frac{4}{5-p} M^{\frac{2(p-1)}{5-p}}$ into \eqref{Pohozaev identities}, we obtain 
\begin{equation}\label{Pohozaev identities'}
\|\partial_x Q_\infty\|_{L^2(\mathbb{R})}^2=\left[\tfrac{(p-1)C_{1,p+1}}{2(p+1)} \right]^\frac{4}{5-p}M^\frac{p+3}{5-p}
\end{equation}
and
\begin{equation}\label{non-relativistic minimum energy}
\mathcal{J}_\infty(M)= -\frac{5-p}{2(p+3)}\mu_\infty M=-\frac{5-p}{2}\left[\tfrac{(p-1)^{\frac{p-1}{4}}C_{1,p+1}}{2(p+1)}\right]^\frac{4}{5-p}M^\frac{p+3}{5-p}<0.
\end{equation}

\section{Existence of a minimizer: Proof of Theorem \ref{cptness}}\label{psm}

We consider the variational problem $\mathcal{J}_c(M)$ (see \eqref{pNLS minimization}) and prove the existence of a minimizer and its basic properties (Theorem \ref{cptness}). As a first step, we show that $\mathcal{J}_c(M)$ has a negative upper bound.
\begin{lemma}[Comparison between $\mathcal{J}_c(M)$ and $\mathcal{J}_\infty(M)$]\label{minimum energy comparison}  Assume that $p\in [3,5)$ and $c\ge (\alpha^\frac{4}{p+3}M)^\frac{p-1}{5-p}$. Then, $Q_\infty$ is admissible for $\mathcal{J}_c(M)$, and 
$$\mathcal{J}_c(M)\leq \mathcal{E}_c(Q_\infty)\leq \mathcal{E}_\infty(Q_\infty)=\mathcal{J}_\infty(M)<0.$$
\end{lemma}

\begin{proof}
By \eqref{alpha}, \eqref{Pohozaev identities'}, Lemma \ref{symbol} and  the assumptions $p\in [3,5)$  and $c\ge (\alpha^\frac{4}{p+3}M)^\frac{p-1}{5-p}$,  we obtain 
$$\|\sqrt{\mathcal{H}_c}Q_\infty\|_{L^2(\mathbb{R})}\leq\|\partial_x Q_\infty\|_{L^2(\mathbb{R})}=\left[\tfrac{(p-1)C_{1,p+1}}{2(p+1)} \right]^\frac{2}{5-p}M^\frac{p+3}{2(5-p)}\leq \alpha^\frac{2}{5-p}M^\frac{p+3}{2(5-p)}\leq c^{\frac{p+3}{2(p-1)}}.$$
Moreover, by \eqref{non-relativistic minimum energy} and the fact that $\|\sqrt{\mathcal{H}_c}Q_\infty\|_{L^2(\mathbb{R})}\leq\|\partial_x Q_\infty\|_{L^2(\mathbb{R})}$, we obtain  $\mathcal{E}_c(Q_\infty)\leq \mathcal{E}_\infty(Q_\infty)=\mathcal{J}_\infty(M)<0$. This proves the lemma by the definition of $\mathcal{J}_c(M)$. 
\end{proof}

Next, we observe from Corollary \ref{separation of initial data} that the constraint in $\mathcal{J}_c(M)$ can be refined.

\begin{lemma}[Refined constraint minimization]\label{local min structure}
Assume that $p\in [3,5)$. Then, for any $c\geq \max\{(\alpha M)^\frac{p-1}{5-p}, (\alpha^\frac{4}{p+3}M)^\frac{p-1}{5-p}\}$, we have 
\begin{equation}\label{reduced pNLS minimization}
\mathcal{J}_c(M)=\inf\Big\{\mathcal{E}_c(u): \|\sqrt{\mathcal{H}_c}u\|_{L^2(\mathbb{R})}\leq \alpha^{\frac{2}{5-p}} M^\frac{p+3}{2(5-p)}  \textup{ and }\|u\|_{L^2(\mathbb{R})}^2=M\Big\}.
\end{equation}
\end{lemma}

\begin{proof}
Let $\{u_n\}_{n=1}^\infty$ be a minimizing sequence for $\mathcal{J}_c(M)$. For a sufficiently large $n\geq 1$, by Lemma \ref{minimum energy comparison}, the set of admissible functions is not empty and $\mathcal{E}_c(u_n)<0$. Hence, Corollary \ref{separation of initial data} implies that $\|\sqrt{\mathcal{H}_c} u_n\|_{L^2(\mathbb{R})}\leq \alpha^{\frac{2}{5-p}}M^{\frac{p+3}{2(5-p)}}$ holds.
\end{proof}

Now, we can show that a minimizer exists for \eqref{pNLS minimization}.


\begin{proof}[Proof of Theorem \ref{cptness}]
Let $\{u_n\}_{n=1}^\infty$ be a minimizing sequence for $\mathcal{J}_c(M)$. Then, by Lemma \ref{local min structure}, it is uniformly bounded in $H^{1/2}(\mathbb{R})$, and thus, $u_n\rightharpoonup \tilde{u}$ in $H^{1/2}(\mathbb{R})$ up to a subsequence. However, we have 
\begin{equation}\label{po1}
\liminf_{n\rightarrow \infty}\|u_n\|_{L^{p+1}(\mathbb{R})}>0,
\end{equation}
because, by Lemma \ref{minimum energy comparison}, $0> \mathcal{J}_c(M)=\mathcal{E}_c(u_n)-o_n(1)\geq-\frac{1}{p+1}\|u_n\|_{L^{p+1}(\mathbb{R})}^{p+1}-o_n(1)$. Here, $o_n(1)$ means that $o_n(1)=a_n\rightarrow 0$ as $n\rightarrow \infty$, where $a_n\in \R$. Hence, $\tilde{u}\neq 0$ in $L^{p+1}(\mathbb{R})$. 

We claim that $u_n\to \tilde{u}$ in $L^2(\mathbb{R})$. If the claim is not true, then passing to a subsequence, 
$$\|\tilde{u}\|_{L^2(\mathbb{R})}^2=M',\quad \|u_n-\tilde{u}\|_{L^2(\mathbb{R})}^2\to M-M'\in (0,M).$$
and thus, 
$$\mathcal{E}_c(u_n)=\mathcal{E}_c(\tilde{u})+\mathcal{E}_c(u_n-\tilde{u})+o_n(1)\geq \mathcal{J}_c(M')+\mathcal{J}_c(M-M')-o_n(1).$$
Let $\{v_n\}_{n=1}^\infty$ be a minimizing sequence for $\mathcal{J}_c(M')$. Then, as observed above (see \eqref{po1}), passing to a subsequence, $\underset{n\to\infty}\liminf\|v_n\|_{L^{p+1}(\mathbb{R})}^{p+1}\geq \delta_0>0$ for a small $\delta_0>0$. Furthermore, $\sqrt{\frac{M}{M'}}v_n$ is admissible for $\mathcal{J}_c(M)$, because Lemma \ref{local min structure} implies that $\|\sqrt{\mathcal{H}_c}(\sqrt{\frac{M}{M'}}v_n)\|_{L^2(\mathbb{R})}^2\leq \frac{M}{M'}\alpha^{\frac{4}{5-p}} (M')^{\frac{p+3}{5-p}}\leq \alpha^{\frac{4}{5-p}}M^{\frac{p+3}{5-p}}$ and $\|\sqrt{\frac{M}{M'}}v_n\|_{L^2(\mathbb{R})}^2=M$. Therefore, it follows that
$$\begin{aligned}
\mathcal{J}_c(M)&\leq \mathcal{E}_c(\sqrt{\tfrac{M}{M'}}v_n)=\frac{M}{M'}\mathcal{E}_c(v_n)-\frac{1}{p+1}\frac{M}{M^\prime}\left[\left(\frac{M}{ M'}\right)^\frac{p-1}{2}-1\right]\|v_n\|_{L^{p+1}(\mathbb{R})}^{p+1}\\
&\leq\frac{M}{M'}\mathcal{J}_c(M')-\frac{1}{p+1}\frac{M}{M^\prime}\left[\left(\frac{M}{ M'}\right)^\frac{p-1}{2}-1\right]\delta_0-o_n(1),
\end{aligned}$$
and thus,
$$\mathcal{J}_c(M')\geq \frac{M'}{M}\mathcal{J}_c(M)+\frac{1}{p+1}\left[\left(\frac{M}{ M'}\right)^\frac{p-1}{2}-1\right]\delta_0-o_n(1).$$
Moreover, by switching the roles of $M'$ and $M-M'$, we have
$$\mathcal{J}_c(M-M')\geq \frac{M-M'}{M}\mathcal{J}_c(M)+\frac{1}{p+1}\left[\left(\frac{M}{M-M'}\right)^\frac{p-1}{2}-1\right]\delta_0-o_n(1)$$
for sufficiently small $\delta_0>0$. In combination, we obtain
$$\mathcal{J}_c(M)\geq\mathcal{J}_c(M)+\frac{1}{p+1}\left[\left(\frac{M}{ M'}\right)^\frac{p-1}{2}+\left(\frac{M}{M-M'}\right)^\frac{p-1}{2}-2\right]\delta_0 -o_n(1),$$
which deduces a contradiction. 
 
We now show that $u_n\rightarrow \tilde{u}$ in $H^{1/2}(\mathbb{R})$, and the limit $\tilde{u}$ is a minimizer for $\mathcal{J}_c(M)$. Indeed, $\tilde{u}$ is admissible because $u_n\to \tilde{u}$ in $L^2(\mathbb{R})$. Moreover, by applying the Gagliardo-Nirenberg inequality (Proposition \ref{GN type inequality}) to $(u_n-\tilde{u})$, we can show that $u_n\to \tilde{u}$ in $L^{p+1}(\mathbb{R})$. Thus, owing to the weak lower semi-continuity of the kinetic energy, it follows that $\mathcal{J}_c(M)=\mathcal{E}_c(u_n)+o_n(1)\geq \mathcal{E}_c(\tilde{u})+\frac12\|\sqrt{\mathcal{H}_c}(u_n-\tilde{u})\|_{L^2(\mathbb{R})} +o_n(1)$. By minimality, it is shown that $\mathcal{E}_c(\tilde{u})=\mathcal{J}_c(M)$ and $u_n\rightarrow \tilde{u}$ in $H^{1/2}(\mathbb{R})$.

It remains to show that $\tilde{u}$ is non-negative up to the phase shift, decreasing, and symmetric. 
Let $\tilde{u}^*$ be the symmetric rearrangement of $\tilde{u}$. Then, $\|\sqrt{\mathcal{H}_c}\tilde{u}^*\|_{L^{2}(\mathbb{R})}\leq\|\sqrt{\mathcal{H}_c}\tilde{u}\|_{L^{2}(\mathbb{R})}$, where the equality holds only if $\tilde{u}=e^{i\theta}\tilde{u}^*(\cdot-x_0)$ for some $\theta\in\mathbb{R}$ and $x_0\in\mathbb{R}$ (see \cite[Appendix A]{LY} or the proof of \cite[Proposition 2.1]{SH} for details). Moreover, we have $\|\tilde{u}^*\|_{L^{p+1}(\mathbb{R})}=\|\tilde{u}\|_{L^{p+1}(\mathbb{R})}$ and $\|\tilde{u}^*\|_{L^2(\mathbb{R})}^2=\|\tilde{u}\|_{L^2(\mathbb{R})}^2=M$. Thus, symmetrization strictly decreases the energy, and a minimizer must be the symmetrization of itself. 
\end{proof}

\section{Non-relativistic limit: proof of Proposition \ref{prop: non-relativistic limit}}\label{nlim}
In this section, we prove the non-relativistic limit. First, we show that relativistic minimizers are uniformly more regular according to a standard argument.

\begin{lemma}[Elliptic regularity]\label{elliptic regularity}
Let $p\in [3,5)$. Then, there exists $c_0\geq 1$ such that $\underset{c\geq c_0}\sup\|Q_c\|_{H^2(\mathbb{R})}<\infty$, where $Q_c$ is a minimizer constructed in Theorem \ref{cptness}.
\end{lemma}

\begin{proof}
Let $c\geq c_0$ be a sufficiently large $c_0\geq 1$. The implicit constants below are independent of $c\geq c_0$. Note that $\mu_c\geq -\frac{p+1}{M}\mathcal{J}_\infty(M)$, where $\mu_c$ is the Lagrange multiplier for $\mathcal{H}_cQ_c-Q_c^p=-\mu_cQ_c$, since by Lemma \ref{minimum energy comparison},
\begin{equation}\label{largr1}
\begin{aligned}
\mathcal{J}_\infty(M)&\geq\mathcal{E}_c(Q_c)=\frac{1}{2}\|\sqrt{\mathcal{H}_c}Q_c\|_{L^2(\mathbb{R})}^2-\frac{1}{p+1}\|Q_c\|_{L^{p+1}(\mathbb{R})}^{p+1}\\
&\geq\frac{1}{p+1} \int_{\mathbb{R}}(\mathcal{H}_cQ_c-Q_c^p)Q_c dx=-\frac{\mu_c}{p+1}\|Q_c\|_{L^2(\mathbb{R})}^2=-\frac{M}{p+1}\mu_c.
\end{aligned}
\end{equation}
Hence, by Lemma \ref{symbol} and the claim, $\|(\mathcal{H}_c+\mu_c)^{-1}f\|_{L^2(\mathbb{R})}\lesssim\|f\|_{H^{-1}(\mathbb{R})}$. Thus, by Lemma \ref{symbol} and Lemma \ref{local min structure}, we obtain 
\begin{align*}
\|Q_c\|_{\dot{H}^1(\mathbb{R})}&=\|(\mathcal{H}_c+\mu_c)^{-1}Q_c^p\|_{\dot{H}^1(\mathbb{R})}\lesssim \|Q_c^p\|_{L^2(\mathbb{R})}\\
&\lesssim  \|Q_c\|_{H^\frac12(\mathbb{R})}^p\lesssim (\|\sqrt{\mathcal{H}_c}Q_c\|_{L^{2}(\mathbb{R})}+\|Q_c\|_{L^2(\mathbb{R})})^p\leq (\alpha^{\frac{2}{5-p}} M^\frac{p+3}{2(5-p)}+\sqrt{M})^p,
\end{align*}
and hence
$$\|Q_c\|_{\dot{H}^2(\mathbb{R})}=\|(\mathcal{H}_c+\mu_c)^{-1}Q_c^p\|_{\dot{H}^2(\mathbb{R})}\lesssim \|Q_c^p\|_{\dot{H}^1(\mathbb{R})}\lesssim \|Q_c\|_{L^\infty(\mathbb{R})}^{p-1} \|Q_c\|_{\dot{H}^1(\mathbb{R})}\lesssim \|Q_c\|_{H^1(\mathbb{R})}^p.$$
\end{proof}

\begin{proof}[Proof of Proposition \ref{prop: non-relativistic limit}]
To prove the non-relativistic limit $Q_c\to Q_\infty$ in $H^1(\R)$, by the concentration-compactness property of the well-known variational problem $\mathcal{J}_\infty(M)$ (see \cite{Lions}) and by the uniqueness of the minimizer $Q_\infty$ (see \cite{K}), it is sufficient to show that $\{Q_c\}_{c\geq c_0}$ is a minimizing sequence for $\mathcal{J}_\infty(M)$, where $c_0\geq1$ is a large number chosen in Lemma \ref{elliptic regularity}. Indeed, we have 
$$\mathcal{E}_\infty(Q_c)=\mathcal{E}_c(Q_c)+\frac{1}{2}\left\{\|\partial_x Q_c\|_{L^2(\mathbb{R})}^2-\|\sqrt{\mathcal{H}_c} Q_c\|_{L^2(\mathbb{R})}^2\right\},$$
but by the Plancherel theorem, 
\begin{equation}\label{kinetic energy comparison}
\begin{aligned}
\|\partial_x Q_c\|_{L^2(\mathbb{R})}^2-\|\sqrt{\mathcal{H}_c} Q_c\|_{L^2(\mathbb{R})}^2&=\frac{1}{2\pi}\int_{\mathbb{R}}\left\{|\xi|^2-\Big(\sqrt{c^2|\xi|^2+\tfrac{c^4}{4}}-\tfrac{c^2}{2}\Big)\right\}|\hat{Q}_c(\xi)|^2d\xi\\
&=\frac{1}{2\pi}\int_{\mathbb{R}}\frac{|\xi|^4}{\sqrt{c^2|\xi|^2+\tfrac{c^4}{4}}+\tfrac{c^2}{2}+|\xi|^2}|\hat{Q}_c(\xi)|^2d\xi\\
&\leq\frac{1}{2\pi c^2}\||\xi|^2\hat{Q}_c\|_{L^2(\mathbb{R})}^2=\frac{1}{c^2}\|Q_c\|_{\dot{H}^2(\mathbb{R})}^2.
\end{aligned}
\end{equation}
Thus, by Lemma \ref{elliptic regularity}, $\mathcal{J}_\infty(M)\leq\mathcal{E}_\infty(Q_c)=\mathcal{E}_c(Q_c)+o_c(1)=\mathcal{J}_c(M)+o_c(1)$, and Lemma \ref{minimum energy comparison} proves that $\{Q_c\}_{c\geq c_0}$ is a minimizing sequence for $\mathcal{J}_\infty(M)$. Here, $o_c(1)$ means that $o_c(1)=a_c\rightarrow 0$ as $c\rightarrow \infty$, where $a_c\in \R$.

Finally, by equation \eqref{kinetic energy comparison} and the non-relativistic limit $Q_c\to Q_\infty$ in $H^1(\mathbb{R})$, we prove that
$$\begin{aligned}
(\mu_c-\mu_\infty)M&=\mu_c \| {Q}_c\|_{L^2(\mathbb{R})}^2-\mu_\infty \| {Q}_\infty\|_{L^2(\mathbb{R})}^2\\
&=-\|\sqrt{\mathcal{H}_c} {Q}_c\|_{L^2(\mathbb{R})}^2 +\|Q_c\|_{L^{p+1}(\R)}^{p+1}+\|\partial_x Q_\infty\|_{L^2(\mathbb{R})}^2 -\|Q_\infty\|_{L^{p+1}(\R)}^{p+1}\to 0.
\end{aligned}$$ 
\end{proof}

\section{Uniqueness: Proof of Theorem \ref{thm: uniqueness}}\label{unq} 
%

We establish the uniqueness of a pseudo-relativistic minimizer, provided that $c\geq 1$ is sufficiently large. A key component is the property of the linearized operator 
$$\mathcal{L}_\infty=-\partial_x^2-pQ_\infty^{p-1}+\mu_\infty,$$
where $Q_\infty$ is a unique positive, symmetric, decreasing minimizer for $\mathcal{J}_\infty(M)$. We denote the set of collections of symmetric $H^s(\mathbb{R})$ functions by $ H _{rad}^ s(\R) $.

\begin{lemma}[Weinstein \cite{We}]\label{limit linearized operator}
For $p\in(1,5)$, there exists $C_0>0$ such that 
$$\langle\mathcal{L}_\infty v, v\rangle_{L^2(\mathbb{R})}\geq C_0\|v\|_{H^1(\mathbb{R})}^2$$
for all $v\in H_{rad}^1(\mathbb{R})$ such that $\langle v,Q_\infty\rangle_{L^2(\mathbb{R})}=0$.
\end{lemma}

We prove that a similar lower bound holds for the linearized operator 
$$\mathcal{L}_c=-\mathcal{H}_c-pQ_c^{p-1}+\mu_c,$$
where $Q_c$ is a minimizer for $\mathcal{J}_c(M)$.

\begin{lemma}\label{eigv}
If $3\leq p<5$, then there exists $C>0$, independent of sufficiently large $c\geq 1$, such that 
$$\langle\mathcal{L}_c v, v\rangle_{L^2(\mathbb{R})}\geq C\|v\|_{H^{1/2}(\mathbb{R})}^2$$
for all $v\in H_{rad}^{1/2}(\mathbb{R})$ such that $\langle v,Q_c\rangle_{L^2(\mathbb{R})}=0$.
\end{lemma}

\begin{proof}
It is sufficient to show that
\begin{equation}\label{eigv proof}
\langle\mathcal{L}_c v, v\rangle_{L^2(\mathbb{R})}\geq \tilde{C}\|v\|_{L^2(\mathbb{R})}^2
\end{equation}
for all $v\in H_{rad}^{1/2}(\mathbb{R})$, such that $\langle v,Q_c\rangle_{L^2(\mathbb{R})}=0$, where $\tilde{C}>0$ is a constant independent of sufficiently large $c\geq 1$. Indeed, if it is true but there is $\{v_c\}_{c\geq c_0}\subset H_{rad}^{1/2}(\mathbb{R})$ such that $\langle v_c, Q_c \rangle_{L^2(\R)}=0$, $\|v_c\|_{H^{1/2}(\R)}=1$ but $\langle \mathcal{L}_c v_c,v_c\rangle_{L^2(\R)}\rightarrow 0$, then by \eqref{eigv proof}, $\| v_c\|_{L^2(\R)}=\tilde{C}^{-1}\langle\mathcal{L}_c v_c, v_c\rangle_{L^2(\mathbb{R})}\rightarrow 0$, and thus $v_c\rightharpoonup 0$ in $H^{1/2}(\R)$. Thus, we have
  $$
o_c(1)=\langle \mathcal{L}_c v_c,v_c\rangle_{L^2(\R)}=\langle  (\mathcal{H}_c +\mu_c)v_c,v_c\rangle_{L^2(\R)} -p\langle Q_c^{p-1} v_c,v_c\rangle_{L^2(\R)}\geq \|v_c\|_{H^{1/2}(\R)}^2+o_c(1). 
$$
which deduces a contradiction.

We suppose that $c\geq 1$ is sufficiently large and define
\begin{equation}\label{relativistic linearized variational problem}
\lambda_c:=\inf\Big\{\langle\mathcal{L}_c v, v\rangle_{L^2(\mathbb{R})}:\ v\in H_{rad}^{1/2}(\mathbb{R}),\ \|v\|_{L^2(\mathbb{R})}=1,\ \langle v,Q_c\rangle_{L^2(\mathbb{R})}=0\Big\}.
\end{equation}
Let $C_0$ be a constant in Lemma \ref{limit linearized operator}, and let $\mu_\infty$ be the Lagrange multiplier for the elliptic equation for the non-relativistic minimizer $Q_\infty$. If $\lambda_c\geq\frac{1}{2}\min\{C_0, \mu_\infty\}$, then \eqref{eigv proof} follows. We assume that  $\lambda_c<\frac{1}{2}\min\{C_0, \mu_\infty\}$. Let $\{v_{c,n}\}_{n=1}^\infty$ be a minimizing sequence for \eqref{relativistic linearized variational problem}. Then, it is uniformly bounded in $H^{1/2}(\mathbb{R})$. Hence, $v_{c,n}\rightharpoonup \tilde{v}_c$ in $H^{1/2}(\mathbb{R})$.

If $\tilde{v}_c\equiv0$, then
$$\begin{aligned}
\lambda_c+o_n(1)&=\langle\mathcal{L}_c v_{c,n}, v_{c,n}\rangle_{L^2(\mathbb{R})}\\
&\geq \mu_c\|v_{c,n}\|_{L^2(\mathbb{R})}^2-p\langle Q_c^{p-1} v_{c,n}, v_{c,n}\rangle_{L^2(\mathbb{R})}\\
&\geq \mu_c-o_n(1)=\mu_\infty-o_c(1)-o_n(1),
\end{aligned}$$
which contradicts $\lambda_c< \frac{\mu_\infty}{2}$.

Suppose that $\tilde{v}_c\neq 0$. Then, by replacing $\tilde{v}_c$ with its normalization $\frac{\tilde{v}_c}{\|\tilde{v}_c\|_{L^2(\mathbb{R})}}$, but still denoted by $\tilde{v}_c$, we obtain a minimizer $\tilde{v}_c$ for \eqref{relativistic linearized variational problem}. Then, because \eqref{relativistic linearized variational problem} is a two-constraint minimization problem, $\tilde{v}_c$ must obey the linear elliptic equation
$$\mathcal{L}_c \tilde{v}_c=\lambda_c \tilde{v}_c+\tilde{\lambda}_c Q_c$$
for some Lagrange multipliers, $\lambda_c$ and $\tilde{\lambda}_c$. Thus, by applying the standard elliptic regularity argument in the proof of Lemma \ref{elliptic regularity}, we can show that $\sup_{c\ge1} \|\tilde{v}_c\|_{H^2(\mathbb{R})}\leq {C}'$, where $C'$ is independent of $c\geq 1$. Now, we let
$$\begin{aligned}
V_c&=\tilde{v}_c-\frac{\langle \tilde{v}_c, Q_\infty\rangle_{L^2(\mathbb{R})}}{\|Q_\infty\|_{L^2(\mathbb{R})}^2}Q_\infty\\
&=\tilde{v}_c-\frac{\langle \tilde{v}_c, Q_c\rangle_{L^2(\mathbb{R})}}{M}Q_c-\frac{\langle \tilde{v}_c, Q_\infty-Q_c\rangle_{L^2(\mathbb{R})}}{M}Q_c-\frac{\langle \tilde{v}_c, Q_\infty\rangle_{L^2(\mathbb{R})}}{M}(Q_\infty-Q_c), \\
&=\tilde{v}_c+R_c,
\end{aligned}$$
because $\langle \tilde{v}_c, Q_c\rangle_{L^2(\mathbb{R})}=0$ and $Q_c\to Q_\infty$ in $H^1(\mathbb{R})$ by the non-relativistic limit (Proposition \ref{prop: non-relativistic limit}). Here, $R_c$ denotes a  function that converges to zero in $H^1(\mathbb{R})$ as $c\rightarrow \infty$. Then  by  \eqref{kinetic energy comparison}, Proposition \ref{prop: non-relativistic limit}, Lemma \ref{limit linearized operator} and the facts that $\langle V_c, Q_\infty\rangle_{L^2(\mathbb{R})}=0$ and $\sup_{c\ge1} \|\tilde{v}_c\|_{H^2(\mathbb{R})}\leq {C}'$, we prove that 
$$C_0\leq \frac{\langle\mathcal{L}_\infty V_c, V_c\rangle_{L^2(\mathbb{R})}}{\|V_c\|_{L^2(\mathbb{R})}^2}=\frac{\langle\mathcal{L}_c\tilde{v}_c, \tilde{v}_c\rangle_{L^2(\mathbb{R})}+o_c(1)}{\|\tilde{v}_c\|_{L^2(\mathbb{R})}^2+o_c(1)}=\lambda_c+o_c(1).$$
\end{proof}

\begin{proof}[Proof of Theorem \ref{thm: uniqueness}]

For a contradiction, we assume that the variational problem $\mathcal{J}_c(M)$ has two different minimizers $Q_c$ and $\tilde{Q}_c$, which are non-negative, symmetric and decreasing, and that $Q_c$ solves the Euler-Lagrange equation $\mathcal{H}_cQ_c-Q_c^p=-{\mu}_c Q_c$ for some Lagrange multiplier $\mu_c>0$. We decompose 
$$\tilde{Q}_c=\sqrt{1-\delta_c^2}Q_c+\epsilon_c, \quad \epsilon_c\neq 0,$$
where $\langle \epsilon_c, Q_c\rangle_{L^2(\R)}=0$. Note that 
\begin{equation}\label{epsilon delta relation}
\delta_c=\frac{1}{\sqrt{M}}\|\epsilon_c\|_{L^2(\R)}\rightarrow 0,
\end{equation}
because $M=\mathcal{M}(\tilde{Q}_c)=(1-\delta_c^2)M+\|\epsilon_c\|_{L^2(\mathbb{R})}^2$ and $(1-\sqrt{1-\delta_c^2})^2M+\|\epsilon_c\|_{L^2(\mathbb{R})}^2=\|\tilde{Q}_c-Q_c\|_{L^2(\mathbb{R})}^2\to 0$ by the non-relativistic limit.

We introduce the functional 
$$\mathcal{I}_c(u)=\mathcal{E}_c(u)+\frac{{\mu}_c}{2}\mathcal{M}(u).$$ 
Clearly, we have  $\mathcal{I}_c({Q}_c)=\mathcal{I}_c(\tilde{Q}_c)=\mathcal{J}_c(M)+\frac{\mu_c}{2}M$. 
Furthermore, by the uniform boundedness of $Q_c$ and $\tilde{Q}_c$, we have
$$\begin{aligned} 
\mathcal{I}_c(\tilde{Q}_c)&=\frac{1}{2}\big\|\sqrt{\mathcal{H}_c+\mu_c}(\sqrt{1-\delta_c^2}Q_c+\epsilon_c)\big\|_{L^2(\mathbb{R})}^2-\frac{1}{p+1}\|\sqrt{1-\delta_c^2}Q_c+\epsilon_c\|_{L^{p+1}(\mathbb{R})}^{p+1}\\
&=\left\{\frac{1-\delta_c^2}{2}\|\sqrt{\mathcal{H}_c+\mu_c} {Q}_c\|_{L^2(\mathbb{R})}^2 -\frac{(1-\delta_c^2)^\frac{p+1}{2}}{p+1}\| {Q}_c\|_{L^{p+1}(\mathbb{R})}^{p+1}\right\}\\
&\quad +  \big\langle \sqrt{1-\delta_c^2}(\mathcal{H}_c+\mu_c) Q_c -(1-\delta_c^2)^\frac{p}{2}  Q_c^p,\epsilon_c\big\rangle_{L^2(\R)}\\
&\quad +\frac12  \big\langle \mathcal{H}_c \epsilon_c-p(1-\delta_c^2)^\frac{p-1}{2}Q_c^{p-1}\epsilon_c+\mu_c \epsilon_c,  \epsilon_c \big\rangle_{L^2(\R)}+o_c(1)\| \epsilon_c \|_{H^{1/2}(\R)}^2.
\end{aligned}$$
Then, by applying the equation $(\mathcal{H}_c+\mu_c)Q_c-Q_c^p=0$ to the first two lines, we write 
$$\begin{aligned} 
\mathcal{I}_c(\tilde{Q}_c)&=\left\{\frac{1}{2}\|\sqrt{\mathcal{H}_c+\mu_c} {Q}_c\|_{L^2(\mathbb{R})}^2 -\Big(\frac{(1-\delta_c^2)^\frac{p+1}{2}}{p+1}+\frac{\delta_c^2}{2}\Big)\| {Q}_c\|_{L^{p+1}(\mathbb{R})}^{p+1}\right\}\\
&\quad +  \big\langle (\sqrt{1-\delta_c^2} -(1-\delta_c^2)^\frac{p}{2})  Q_c^p,\epsilon_c\big\rangle_{L^2(\R)}\\
&\quad +\frac12  \big\langle \mathcal{H}_c \epsilon_c-p(1-\delta_c^2)^\frac{p-1}{2}Q_c^{p-1}\epsilon_c+\mu_c \epsilon_c,  \epsilon_c \big\rangle_{L^2(\R)}+o_c(1)\| \epsilon_c \|_{H^{1/2}(\R)}^2.
\end{aligned}$$
Thus, it follows from \eqref{epsilon delta relation} that 
$$\mathcal{I}_c(\tilde{Q}_c)=\mathcal{I}_c(Q_c)+\frac12  \big\langle \mathcal{L}_c\epsilon_c+\mu_c \epsilon_c,  \epsilon_c \big\rangle_{L^2(\R)}+o_c(1)\| \epsilon_c \|_{H^{1/2}(\R)}^2.$$
Then, the non-degeneracy of $\mathcal{L}_c$ yields $0\geq \frac{C_0}{2}\|\epsilon_c\|_{H^{1/2}(\R)}^2+o_c(1)\|\epsilon_c\|_{H^{1/2}(\R)}^2$, which contradicts the fact that $\epsilon_c\neq 0$. 
\end{proof}

\appendix 

\section{Energy minimizer as a ground state}\label{sec: ground state}
 In this appendix, we show that although the energy minimizer $Q_c$ constructed in Theorem \ref{cptness} is not a global minimizer, it can be still called a \textit{ground state} in the sense that
$$
\mathcal{E}_c(Q_c)=\inf \Big\{\mathcal{E}_c(u): \mathcal{E}_c|_\mathcal{N}'(u)=0 \Big\},
$$
where $\mathcal{N}=\{v\in H^\frac12(\R) \ : \ \mathcal{M}(v)=M\}$. Indeed, this was shown in \cite{BGV}, but it is sketched here again for the sake of completeness.

We start with the following Pohozaev identity (see  \cite{FL}).
\begin{lemma}[Pohozaev identity]\label{pohole}
If $u\in H^\frac12(\R)$ solves  
\begin{equation}\label{elqs}
\mathcal{H}_cu-|u|^{p-1}u=-\mu u
\end{equation}
for some $\mu>0$, then
\begin{align*}
&\frac12\|\sqrt{\mathcal{H}_c}u\|_{L^2(\R)}^2-\frac{p-1}{2(p+1)}\|u\|_{L^{p+1}}^{p+1}+\frac{1}{4} \int_{\R}\frac{\sqrt{c^2|\xi|^2+\frac{c^4}{4}}-\frac{c^2}{2}}{\sqrt{\frac{1}{4 }+\frac{|\xi|^2}{c^2}}}|\hat{u}(\xi)|^2d\xi=0.
\end{align*}
\end{lemma}
\begin{proof}
The proof of the result was proved in \cite{FL}, but we briefly sketch of the proof for the convenience of the reader.  
We multiply $u$ to   the equation \eqref{elqs} and then integration by parts shows 
\begin{equation}\label{neha}
\|\sqrt{\mathcal{H}_c}u\|_{L^2(\R)}^2+\mu\|u\|_{L^2(\R)}^2-\|u\|_{L^{p+1}}^{p+1}=0.
\end{equation}
On the other hand, from the result of \cite[Proposition 1.1]{FL}, we see that $u\in H^2(\R)$ and $|u|+|x\partial_x u|\le C(1+|x|^2)^{-1}.$ Then, by the Plancherel theorem and elementary calculations, we observe that
\begin{align*}
\left\langle \sqrt{-c^2\partial_x^2+\frac{c^4}{4}}u, x\partial_x u \right\rangle &=\frac{d}{d\lambda}\Bigg|_{\lambda=1}\left\langle \sqrt{-c^2\partial_x^2+\frac{c^4}{4}} u, u(\lambda \cdot) \right\rangle\\
&=\frac{d}{d\lambda}\bigg|_{\lambda=1}\frac{1}{2\pi}\int_{\R}\sqrt{c^2|\xi|^2+\frac{c^4}{4}} \hat{u}(\xi)  \overline{\frac{1}{\lambda}\hat{u}\left(\frac{\xi}{\lambda}\right)}d\xi\\
&=\frac{d}{d\lambda}\bigg|_{\lambda=1}\frac{1}{2\pi}\int_{\R} \sqrt{c^2|\xi|^2+\frac{c^4}{4\lambda}} \hat{u}(\sqrt{\lambda}\xi) \overline{ \hat{u}\left(\frac{\xi}{\sqrt{\lambda}}\right)}d\xi\\
&=-\frac{c^4}{16\pi}\int_{\R}\frac{|\hat{u}(\xi)|^2}{\sqrt{c^2|\xi|^2+\frac{c^4}{4 }}}d\xi
\end{align*}
and
\begin{align*}
\left\langle \left(-\frac{c^2}{2}+\mu\right) u-u^p, x\partial_x u\right\rangle&=\int_{\R}\frac12 \left(-\frac{c^2}{2}+\mu\right)  x\partial_x(u^2)-\frac{1}{p+1}x\partial_x(u^{p+1})dx\\
&=\frac12 \left(\frac{c^2}{2}-\mu\right)\|u\|_{L^2(\R)}^2+\frac{1}{p+1}\|u\|_{L^{p+1}}^{p+1}.
\end{align*}
Hence, combining these with $\frac{c^4}{4\sqrt{c^2|\xi|^2+\frac{c^4}{4 }}}=\sqrt{c^2|\xi|^2+\frac{c^4}{4 }}-\frac{c^2|\xi|^2}{\sqrt{c^2|\xi|^2+\frac{c^4}{4}}}$ and using the equation \eqref{elqs}, we obtain 
\begin{equation}\label{poho1}
-\frac12\|\sqrt{\mathcal{H}_c}u\|_{L^2(\R)}^2-\frac{\mu}{2}\|u\|_{L^2(\R)}^2+\frac{1}{p+1}\|u\|_{L^{p+1}}^{p+1}+\frac{c^2}{2}\int_{\R}\frac{|\xi|^2|\hat{u}(\xi)|^2}{\sqrt{c^2|\xi|^2+\frac{c^4}{4}}}d\xi=0.
\end{equation}
Thus, by \eqref{neha} and \eqref{poho1}, we prove the result.
\end{proof}
Next, we prove that the local minimizer $Q_c$ constructed in Theorem \ref{cptness} is a ground state. 
\begin{proposition}
 Assume that $p\in (3,5)$ and $c\geq \max\{(\alpha M)^\frac{p-1}{5-p}, (\alpha^\frac{4}{p+3}M)^\frac{p-1}{5-p},(\frac{M}{p-3})^\frac{p-1}{5-p}\}$. Then we have
$$
\mathcal{J}_c(M)=\inf \Big\{\mathcal{E}_c(u): \mathcal{E}_c|_\mathcal{N}'(u)=0 \Big\},
$$
where $\mathcal{N}=\{v\in H^\frac12(\R) \ : \ \mathcal{M}(v)=M\}$.
\end{proposition}
\begin{proof}
We assume on the contrary that   there exists a critical point $u$ for $\mathcal{E}_c$ on $\mathcal{N}$ with $\mathcal{E}_c(u)<\mathcal{J}_c(M)$. Then $u\in H^\frac12(\R)$ is a solution of 
$$
\mathcal{H}_cu-|u|^{p-1}u=-\mu u
$$
for some $\mu\in \R$. We observe that by \eqref{largr1} Lemma \ref{minimum energy comparison}, $\mu>0$. Then by 
Lemma \ref{minimum energy comparison} and Lemma \ref{pohole}, we get
\begin{align*}
0&>\mathcal{J}_c(M)>\mathcal{E}_c(u)\\
&=\frac{1}{2}\|\sqrt{\mathcal{H}_c}u\|_{L^2(\mathbb{R})}^2-\frac{1}{p+1}\|u\|_{L^{p+1}(\mathbb{R})}^{p+1}\\
&=\frac{1}{2}\left(1-\frac{2}{p-1}\right)\|\sqrt{\mathcal{H}_c}u\|_{L^2(\mathbb{R})}^2-\frac{c^2}{2(p-1)}\|u\|_{L^2(\R)}^2+\frac{c^4}{4(p-1)}\int_{\R}\frac{|\hat{u}(\xi)|^2}{\sqrt{c^2|\xi|^2+\frac{c^4}{4}}}d\xi\\
&\ge \frac{p-3}{2(p-1)} \|\sqrt{\mathcal{H}_c}u\|_{L^2(\mathbb{R})}^2-\frac{c^2}{2(p-1)}\|u\|_{L^2(\R)}^2,
\end{align*}
From this and the assumption $c\ge (\frac{M}{p-3})^\frac{p-1}{5-p}$, we have $ \|\sqrt{\mathcal{H}_c}u\|_{L^2(\mathbb{R})}\le \sqrt{\frac{M}{p-3}}c\le c^{\frac{p+3}{2(p-1)}},$ which implies that $\mathcal{J}_c(M)\le\mathcal{E}_c(u).$
\end{proof}


\begin{thebibliography}{20}

\bibitem{BGV}
J. Bellazzini, V. Georgiev and N. Visciglia, \emph{Long time dynamics for semi-relativistic NLS and half wave in arbitrary dimension}, Math. Ann. \textbf{371} (2018), no. 1-2, 707–740.

\bibitem{BBJV} J. Bellazzini, N. Boussa\"id, L. Jeanjean and N. Visciglia, \emph{Existence and stability of standing waves for supercritical NLS with a partial confinement}, Comm. Math. Phys. \textbf{353} (2017), no. 1, 229--251.

\bibitem{BR} J.P. Borgna and D.F. Rial, \emph{Existence of ground states for a one-dimensional relativistic Schrödinger equation}, J. Math. Phys. \textbf{53} (2012), no. 6, 062301, 19 pp.

\bibitem{BHL} T. Boulenger, D. Himmelsbach and E. Lenzmann, \emph{Blowup for fractional NLS}, J. Funct. Anal.  \textbf{271} (2016), no. 9, 2569--2603.

\bibitem{Caz} T. Cazenave, \emph{Semilinear Schr\"odinger equations}, Courant Lecture Notes in Mathematics, 10. New York University, Courant Institute of Mathematical Sciences, New York; American Mathematical Society, Providence, RI, 2003. xiv+323 pp.

\bibitem{CL}
T.  Cazenave and P. L. Lions,  \emph{Orbital stability of standing waves for some nonlinear Schrödinger equations}, Comm. Math. Phys. \textbf{ 85 } (1982), no. 4, 549–561.



\bibitem{CHS} W. Choi, Y. Hong and J. Seok, \emph{On critical and supercritical pseudo-relativistic nonlinear Schr\"odinger equations}, Proc. Roy. Soc. Edinburgh Sect. A \textbf{150} (2020), no. 3, 1241--1263.



\bibitem{FL}
R. Frank and E. Lenzmann,   \emph{ Uniqueness of non-linear ground states for fractional Laplacians in $\R$}, Acta Math.  \textbf{210} (2013), no. 2, 261–318.



\bibitem{FJL} J. Fr\"ohlich, B.L.G. Jonsson and E. Lenzmann, \emph{Effective dynamics for boson stars}, Nonlinearity \textbf{20} (2007), no. 5, 1031--1075.

\bibitem{FMO} K. Fujiwara, S. Machihara and T. Ozawa, \emph{Remark on a semirelativistic equation in the energy space}, Discrete Contin. Dyn. Syst. 2015, Dynamical systems, differential equations and applications. 10th AIMS Conference. Suppl., 473--478.


\bibitem{K}
M.K. Kwong, \emph{Uniqueness of positive solutions of $\Delta u-u+u^p=0$ in $\R^n$}, Arch. Rational Mech. Anal. \textbf{105} (1989), no. 3, 243–266. 


\bibitem{Lenzmann0} E. Lenzmann, \emph{Well-posedness for semi-relativistic Hartree equations of critical type}, Math. Phys. Anal. Geom. \textbf{10} (2007), no. 1, 43--64.

\bibitem{Lenzmann} E. Lenzmann, \emph{Uniqueness of ground states for pseudorelativistic Hartree equations}, Anal. PDE \textbf{2} (2009), no. 1, 1--27.


\bibitem{KLR} J. Krieger, E. Lenzmann and P. Rapha\"el, \emph{Nondispersive solutions to the L2-critical half-wave equation}, Arch. Ration. Mech. Anal. \textbf{209} (2013), no. 1, 61--129. 



\bibitem{LY} E.H. Lieb and H.-T. Yau, \emph{The Chandrasekhar theory of stellar collapse as the limit of quantum mechanics}, Comm. Math. Phys. \textbf{112} (1987), no. 1, 147--174.

\bibitem{Lions} P.-L. Lions, \emph{The concentration-compactness principle in the calculus of variations, The locally compact case. II}, Ann. Inst. H. Poincar\'e Anal. Non Lin\'eaire \textbf{1} (1984), no. 4, 223–283.

\bibitem{LY1} X. Luo and T. Yang, \emph{Stable solitary waves for pseudo-relativistic Hartree equations with short range potential}, Nonlinear Anal. \textbf{207} (2021), 112275, 13 pp.



\bibitem{SP} Q. Shi and C. Peng, \emph{Wellposedness for semirelativistic Schr\"odinger equation with power-type nonlinearity}, Nonlinear Anal. \textbf{178} (2019), 133--144.

\bibitem{SH} J. Seok and Y. Hong, \emph{Ground States to the Generalized Nonlinear Schr\"odinger Equations with Bernstein Symbols}, Anal. Theory Appl. \textbf{37} (2021), no. 2, 157--177.

\bibitem{We} M. I. Weinstein,   \emph{Modulational stability of ground states of nonlinear Schrödinger equations}, SIAM J. Math. Anal. \textbf{16} (1985), no. 3, 472--491.

\end{thebibliography}
\end{document}